\documentclass[12pt]{amsart}
\usepackage{amsmath,amssymb,amsfonts,amsthm,amsopn}
\usepackage{graphicx,tikz}
\usepackage[all]{xy}
\usepackage{amsmath}
\usepackage{multirow, longtable, makecell, caption, array,enumitem}
\usepackage{amssymb}
\usepackage{mathtools}
\mathtoolsset{showonlyrefs}

\usepackage{bookmark}
\usepackage{hyperref,color}

\calclayout

\usepackage{geometry}
\newcommand{\adm}[1]{{\left\vert\kern-0.25ex\left\vert\kern-0.25ex\left\vert #1 
		\right\vert\kern-0.25ex\right\vert\kern-0.25ex\right\vert}}

\newtheorem{theorem}{Theorem}[section]
\newtheorem*{theorem*}{Theorem}
\newtheorem{lemma}[theorem]{Lemma}

\newtheorem{proposition}[theorem]{Proposition}

\newtheorem{remark}[theorem]{Remark}

\theoremstyle{definition}

\usepackage{cellspace} %
\let\det\relax 
\DeclareMathOperator{\det}{det}

\def\bR{\mathbb{R}}
\def\bC{\mathbb{C}}
\def\bN{\mathbb{N}}

\def\smo{\setminus\{0\}}

\def\cF{\mathcal{F}}

\def\cK{\mathcal{K}}

\def\cS{\mathcal{S}}

\def\cL{\mathcal{L}}

\def\rd{\bR^d}

\def\rnn{\bR^{2n}}
\def\rnp{\bR^{n'}}
\def\rnnp{\bR^{2n'}}
\def\rns{\bR^{n''}}
\def\rnns{\bR^{2n''}}

\def\rdd{\bR^{2d}}
\def\rdq{\bR^{4d}}

\def\lan{\langle}
\def\ran{\rangle}

\def\wt{\widetilde}
\def\wh{\widehat}

\def\w{\mathrm{w}}

\def\S0{S^0_{0,0}}

\def\Bd'{B_{\delta'}}

\def\cBd'{\bar{B}_{\delta'}}

\def\bots{{\sigma \bot}}
\def\Sbots{S^{\sigma \bot}}

\newcommand\restr[2]{{
		\left.\kern-\nulldelimiterspace 
		#1 
		\right|_{#2} 
}}
\def\a{\alpha}
\def\b{\beta}
\def\g{\gamma}

\def\irp{\int_0^{+\infty}}
\def\ird{\int_{\rd}}
\def\irdd{\int_{\rdd}}

\def\Spdr{\mathrm{Sp}(d,\bR)}
\def\Spdc{\mathrm{Sp}(d,\bC)}

\def\Sp{\mathrm{Sp}}
\newcommand*\dd[1]{\mathop{}\!\mathrm{d}#1}

\def\diag{\mathrm{diag}}
\def\re{\operatorname{Re}}
\def\im{\operatorname{Im}}

\DeclarePairedDelimiter{\abs}{\lvert}{\rvert}
\DeclarePairedDelimiter{\norm}{\lVert}{\rVert}
\DeclarePairedDelimiter\floor{\lfloor}{\rfloor}

\makeatletter
\let\oldabs\abs
\def\abs{\@ifstar{\oldabs}{\oldabs*}}

\makeatletter
\let\oldnorm\norm
\def\norm{\@ifstar{\oldnorm}{\oldnorm*}}

\begin{document}
	
	\title[Wave packet analysis of quadratic semigroups]{Wave packet analysis of semigroups generated by quadratic differential operators}
	
	\author[S. I. Trapasso]{S. Ivan Trapasso}
	\address{Dipartimento di Scienze Matematiche ``G. L. Lagrange'', Politecnico di Torino, corso Duca degli Abruzzi 24, 10129 Torino, Italy}
	\email{salvatore.trapasso@polito.it}
	
	
	\begin{abstract}
We perform a phase space analysis of evolution equations associated with the Weyl quantization $q^\w$ of a complex quadratic form $q$ on $\rdd$ with non-positive real part. In particular, we obtain pointwise bounds for the matrix coefficients of the Gabor wave packet decomposition of the generated semigroup $e^{tq^\w}$ if $\re q \le 0$ and the companion singular space associated is trivial. This result is then leveraged to achieve a comprehensive analysis of the phase regularity of $e^{tq^\w}$ with $\re q \le 0$, thereby extending the $L^2$ analysis of quadratic semigroups initiated by Hitrik and Pravda-Starov to general modulation spaces $M^p(\rd)$, $1 \le p \le \infty$, with optimal explicit bounds. 
	\end{abstract}
	\maketitle
	
\section{Introduction}
\subsection{Quadratic operators}
The topic of this note is the analysis of evolution problems such as
\begin{equation}
	\begin{cases}
		\partial_t u(t,x)=q^\w u(t,x) \\
		u(0,\cdot)=f,
	\end{cases} \qquad (t,x) \in [0,+\infty) \times \rd,
\end{equation} where initially $f \in \cS(\rd)$ and $q^\w(x,D)$ is the Weyl quantization of a complex-valued quadratic form $q$ on $\rdd$, that is
\begin{equation}\label{eq-intro-weylq}
	q^\w (x,D) f(x) = (2\pi)^{-d} \irdd e^{i(x-y)\cdot \xi} q\Big( \frac{x+y}{2}, \xi \Big) f(y) \dd{y} \dd{\xi}. 
\end{equation} Note that this is a genuine differential operator, since for $p(x,\xi)=x^\a \xi^\b$ with $\a,\b \in \bN^d$ such that $\abs{\a+\b}\le 2$ we have
\begin{equation}
	p^\w(x,D) = \frac{x^\a  D^\b + D^\b x^\a}{2}, \qquad D = -i\partial_x. 
\end{equation} 
Moreover, the operator $q^\w$ may or may not be of elliptic type, depending on whether the symbol $q$ satisfies the constraint
\begin{equation}
	(x,\xi) \in \rdd, \quad q(x,\xi) = 0 \quad \Longrightarrow \quad (x,\xi)=0.
\end{equation}

A considerable body of knowledge on the features of the generated solution semigroup $e^{tq^\w}$, $t \ge 0$, has accumulated over the years. For instance, a full spectral picture in the elliptic case has been set out by Sj\"ostrand in a fundamental contribution \cite{sjo_74}. It was then H\"ormander who showed in \cite{horm-mehler-95} that the (non-selfadjoint) maximal closed realization of $q^\w$ on $L^2$ with domain $\{u \in L^2(\rd) : q^\w u \in L^2(\rd)\}$ coincides with the graph closure of its restriction to (an endomorphism of) the Schwartz class $\cS(\rd)$. 

Much more can be said in the case where the real part of the form $q$ has a constant sign. In particular, if $\re q \le 0$ then the generated semigroup $e^{tq^\w}$ is contractive on $L^2(\rd)$, and can be viewed as a Fourier integral operator with Gaussian distribution kernel or as a one-parameter family of pseudo-differential operators whose Weyl symbols can be explicitly computed. Some preparation is needed in order to state the result just mentioned in more precise terms. 

Let $Q \in \bC^{2d,2d}$ be the symmetric matrix naturally associated with $q$ by virtue of the relation $q(x,\xi) = (x,\xi) \cdot Q(x,\xi)$. One can then determine the companion \textit{Hamilton map}, or fundamental matrix, that is $F=JQ \in \bC^{2d,2d}$, where 
\begin{equation}\label{eq-intro-defJ}
	J = J_d = \begin{bmatrix}
		O_d & I_d \\ -I_d & O_d
	\end{bmatrix} \in \bR^{2d,2d}
\end{equation} is the canonical symplectic matrix, $O_d$ and $I_d$ denoting respectively the $d\times d$ null and identity matrix. The matrix $J$ induces the standard symplectic structure on the phase space $T^*\rd \simeq \rdd$ via the form
\begin{equation}
	\sigma(z,w) \coloneqq Jz \cdot w, \qquad z,w \in \rdd,
\end{equation} so that $F$ is uniquely identified by the relation $
	q(z;w) = \sigma(z,Fw)$, where $q(\cdot;\cdot)$ stands for the polarized version of the quadratic form $q$. 
	
	We are now ready to state the result obtained by H\"ormander \cite[Theorem 4.2]{horm-mehler-95}: there exists a family of symbols $\Theta_t \in \cS'(\rdd)$ such that $e^{tq^\w}= \Theta_t^\w$ and, after setting 
	\begin{equation}
		\mathfrak{E} = \{ s \ge 0 : \det(\cos(sF)) = 0\},
	\end{equation} we have the explicit representation
\begin{equation}\label{eq-intro-mehler}
	\Theta_t(z) = (\det^{-1/2} \cos(tF)) \exp(\sigma(z,\tan(tF)z)), \qquad z \in \rdd, \quad t \notin \mathfrak{E}.
\end{equation} This relation is known as the \textit{generalized Mehler formula}, since it encompasses the eponymous expression for the symbol of the quantum harmonic oscillator propagator $e^{-t(x^2 + D^2)}$ first obtained (on other grounds, with $d=1$) in \cite{mehler}, that is $\Theta_t(z) = (\cosh t)^{-d} e^{-(\tanh t) \abs{z}^2}$. 

Our essential review of the literature on quadratic operators cannot avoid mentioning some remarkable contributions in the non-elliptic scenario, which is obviously trickier. In particular, Hitrik and Pravda-Starov showed in \cite{hitrik} that if $\re q \le 0$ then there exists a phase space subspace $S \subseteq T^*\rd$, called the \textit{singular space} associated with $q$, enjoying the following property: if $S$ is a symplectic subspace, then $e^{tq^\w}$ is smoothing along every direction of the symplectic orthogonal complement $\Sbots$. More precisely, the singular space is defined by
\begin{equation}\label{eq-intro-singsp}
	S \coloneqq \Bigg( \bigcap_{j=0}^{2d-1} \ker [\re F(\im F)^j] \Bigg) \cap \rdd.
\end{equation} 
The singular space $S$ is thus the largest subspace where the underlying dissipative and dispersive flows are decoupled, and the former are completely inhibited. Figuratively speaking, $S$ is designed to detect, at the phase space (i.e., classical) level, the mixing dynamics possibly arising from non-commutativity between $(\re q)^\w$ and $(\im q)^\w$.

Note that the assumption on the (inherited) symplectic structure of $S$ is trivially satisfied in the case where the symbol $q$ enjoys a restricted ellipticity condition, that is 
\begin{equation}
	(x,\xi) \in S, \quad q(x,\xi) = 0 \quad \Longrightarrow \quad (x,\xi)=0.
\end{equation}
In fact, this condition suffices to prove a non-elliptic counterpart of the aforementioned results by Sj\"ostrand on the spectral structure of $e^{tq^\w}$. In addition, the analysis carried out in \cite[Theorem 1.2.3]{hitrik} on the FBI-Bargmann transform side shows that the quadratic semigroup enjoys exponential decay like
\begin{equation}
	\| e^{tq^\w} \|_{L^2 \to L^2} \le Ce^{-at}, \qquad t \ge 0,
\end{equation} for suitable $C,a>0$ independent of $t$, if and only if $q$ is not purely imaginary, which is in turn equivalent to have a  non-degenerate (yet symplectic) singular space $S \ne \rdd$. 

In the interesting case of zero singular space, viz.\ $S= \{0\}$, the optimal rate of exponential decay in the $L^2\to L^2$ bound is deduced from the spectral analysis carried out in \cite{ottobre}. More recently, again in the case $S=\{0\}$, the same circle of ideas and techniques (involving metaplectic FBI transforms and Fourier integral operators) allowed obtaining small time asymptotic $L^2 \to L^\infty$ bounds in \cite{hitrik_18}, while more general $L^p \to L^q$ bounds with $1 \le p \le q \le \infty$ were proved in \cite{white_22} in both the small and large time regimes. 

\subsection{The point of view of wave packet analysis}

The previous discussion suggests very clearly that the analysis of quadratic operators and related semigroups can benefit from ideas and techniques of phase space analysis. In this note we take a different angle on the matter and perform phase space analysis in the sense of wave packets decompositions. To be precise, a Gabor wave packet generated by a given $g \in \cS(\rd)\smo$ is a function of the form 
\begin{equation}
	\pi(z) g(y) \coloneqq (2\pi)^{-d/2} e^{i \xi \cdot y} g(y-x), \qquad z=(x,\xi)\in \rdd.
\end{equation} Intuitively speaking, the family $\{\pi(z)g : z \in \rdd\}$ induces a continuous, uniform covering of phase space by means of shifts along $(x,\xi) \in T^*\rd$, which in turn triggers a continuous wave packet decomposition of a temperate distribution $f \in \cS'(\rd)$, the coefficients being given by
\begin{equation}
	V_g f(z) \coloneqq \lan f,\pi(z)g \ran = (2\pi)^{-d/2} \ird e^{-i \xi \cdot y} f(y) \overline{g(y-x)} \dd{y}.
\end{equation} 
On the other hand, the correspondence $T^*\rd \ni z \mapsto \lan f, \pi(z)g \ran \in \bC$ can be rightfully viewed as a phase space representation of $f$, usually called the \textit{Gabor transform} of $f$, which encodes its joint time-frequency features. For instance, the study of the regularity of a function can be lifted to the phase space level, where the (possibly mixed, weighted) summability of the corresponding Gabor transform provides finer information. This is exactly the rationale behind the introduction of \textit{modulation spaces} \cite{fei_81}: for $1 \le p \le \infty$, the space $M^p(\rd)$ is the collection of $f \in \cS'(\rd)$ such that, for some (in fact, any) $g \in \cS(\rd)\smo$, 
\begin{equation}
	\norm{f}_{M^p} \coloneqq \norm{V_g f}_{L^p} = \Big( \irdd \abs{V_g f(z)}^p \dd{z} \Big)^{1/p} < \infty. 
\end{equation}
It turns out that modulation spaces $(M^p(\rd),\norm{\cdot}_{M^p})$ are a family of Banach spaces, increasing with $p$, a distinguished member being $L^2(\rd) = M^2(\rd)$ (with equivalent norms). They thus offer a manageable complement to the classical triple of Fourier analysis $\cS \hookrightarrow L^2 \hookrightarrow \cS'$, especially when it comes to the analysis of (pseudo-)differential operators. 

In general terms, the wave packet analysis of a linear operator $T \colon \cS(\rd) \to \cS'(\rd)$ revolves around the in-depth study of its \textit{Gabor matrix} --- that is, given $g,\gamma \in \cS(\rd)\smo$, the quantity
\begin{equation}
	K_{T}^{(\gamma,g)}(w,z) \coloneqq \lan T \pi(z)g, \pi(w) \gamma \ran, \qquad z,w \in \rdd. 
\end{equation} It is fairly evident that this kernel encodes the whole information on the action of $T$ at the wave packet level. Indeed, it can be used to lift the analysis of operators to phase space via the following identity:
\begin{equation}\label{eq-intro-lens}
	V_\gamma(T f)(w) = \irdd K_{T}^{(\gamma,g)}(w,z) V_g f(z) \dd{z}, 
\end{equation} where it is assumed that $\norm{g}_{L^2}=1$ for convenience (cf.\ \eqref{eq-inv-stft} below). Seminal contributions in connection with this approach are due to Tataru \cite{tataru}, while more recent and diverse applications can be found in \cite{CR_book,gro_book,NT_book}. 

As far as quadratic operators are concerned, the point of view of Gabor analysis has been widely employed in the context of semigroups generated by purely imaginary quadratic forms, that is $e^{tq^\w} = e^{it (\im q)^\w}$. These are well known examples of \textit{metaplectic operators} \cite{degoss_11_book,folland}, and can be equivalently viewed as Schr\"odinger propagators for quantum systems with quadratic Hamiltonians --- so that wave packets provide natural and meaningful models to test the features of the corresponding unitary dynamics and related phenomena of dispersive nature. 

In accordance with this spirit, it has been recently proved in \cite{CNT_dispdiag20} that the Gabor matrix of a metaplectic operator is particularly well-structured, as evidenced by the following pointwise bound: for every $N \in \bN$ and $z,w \in \rdd$, 
\begin{equation}\label{eq-intro-metap-dispdiag}
	\abs{\lan e^{it(\im q)^\w} \pi(z)g, \pi(w) \gamma \ran} \le C (\sigma_1(t)\cdots \sigma_d(t))^{-1/2} (1+\abs{M_t(w-H_tz)})^{-N},
\end{equation} for a suitable $C>0$ that does not depend on $t$ nor $q$. We emphasize that all the relevant phenomena characterizing the metaplectic evolution become noticeable at this level (see also Section \ref{sec-metap} for additional details):
\begin{itemize}
	\item The Gabor matrix enjoys fast decay away from the graph of the symplectic flow $t \mapsto H_t=e^{2tJ(\im Q)}$ --- that is, Gabor wave packets are approximately evolved by the quantum propagator $e^{it(\im q)^\w}$ along the phase space trajectories of the corresponding classical system. 
	
	\item The factors $\sigma_1(t) \ge \cdots \ge \sigma_d(t) \ge 1$ are the largest singular values of $H_t$, which account for dispersive phenomena \cite{cauli}. For instance, in the free particle case (viz., $q(x,\xi)=i\xi^2$) the product $(\sigma_1(t)\cdots\sigma_d(t))^{-1/2}$ coincides, up to irrelevant constants, with the standard dispersive factor $(1+\abs{t})^{-d/2}$.  
	
	\item The matrix $M_t \in \bR^{2d,2d}$ produces a phase space envelope along the classical trajectories that has been linked to the well known quantum spreading phenomenon incurred by wave packets \cite{dirac}.  
\end{itemize}
Results of this type can be then leveraged to investigate boundedness on modulation spaces. To put it simply, the bound in \eqref{eq-intro-metap-dispdiag} and the relation \eqref{eq-intro-lens} imply that $e^{it(\im q)^\w}$ roughly acts as a convolution operator in phase space, hence continuity on modulation spaces follows (see \cite[Corollary 3.5]{CNT_dispdiag20}): 
\begin{equation}\label{eq-intro-metbound}
	\norm{e^{it(\im q)^\w}f}_{M^p} \le C (\sigma_1(t)\cdots \sigma_d(t))^{\abs{\frac1p - \frac12}} \norm{f}_{M^p}.
\end{equation} This result shows that the Schr\"odinger propagator $e^{it(\im q)^\w}$ preserves the time-frequency regularity/concentration of the initial datum, while similar conclusions are known to be false in $L^p(\rd)$ unless $p=2$. Note also that the unitary $L^2$ dynamic is blind to the dispersive effects occurring on other modulation spaces. 

It might seem quite surprising that such a comprehensive Gabor analysis of metapletic operators does not have a counterpart in the case where $\re q \ne 0$. Indeed, except for isolated cases (e.g., the standard heat equation \cite{chen}, the Ginzburg-Landau model in \cite{wang} or the semilinear parabolic models in \cite{nicola_14}), a systematic approach to dissipative semigroups from the wave packet analysis point of view appears to be missing. For instance, the first results on the modulation space regularity of solutions of the heat equation in the case of the Hermite operator (that is, the quantum harmonic oscillator with elliptic symbol $q(x,\xi)=-(x^2+\xi^2)$) have been obtained only recently in \cite{bhimani_23, bhimani_21,cordero_21}. In particular, the analysis of \cite{cordero_21} shares the spirit described so far, since boundedness results occur as a byproduct of almost-diagonalization estimates for the Gabor matrix of $e^{tq^\w}$ such as 
\begin{equation}\label{eq-intro-hermite-dec}
	\abs{\lan e^{-t(x^2+D^2)} \pi(z)g, \pi(w)\gamma \ran} \le C(T) (1+\abs{w-z})^{-N}, \qquad N \in \bN,
\end{equation} for every given $T>0$ and $t\in [0,T]$. Although this is enough to obtain continuity results via Young's inequality in \eqref{eq-intro-lens}, such estimates fail to explicitly detect the exponential decay over time that could be naturally expected from the already mentioned results by Hitrik and Pravda-Starov in \cite{hitrik}. These difficulties are circumvented in \cite{bhimani_21}, where the sharp time decay is obtained:
\begin{equation}\label{eq-intro-expdec}
	\norm{e^{-t(x^2+D^2)}}_{M^p \to M^p} \le Ce^{-td}, \quad t \ge 0, 
\end{equation} although by means of a significantly different approach that involves the spectral theory and pseudo-differential calculus of globally elliptic positive operators. On the other hand, when it comes to transfer these results to the Gabor matrix, duality and embeddings of modulation spaces imply
	\begin{equation}
		\abs{\lan e^{-t(x^2+D^2)} \pi(z)g, \pi(w)\gamma \ran} \le C e^{-td} \norm{g}_{M^1}\norm{\g}_{M^1},
	\end{equation} thus failing now to detect the fact that Gabor wave packets approximately diagonalize every Weyl operator with a smooth bounded symbol with bounded derivatives of any order --- see \cite{bastianoni,gro-rze_08,tataru}, also for more general H\"ormander symbol classes.  

A deeply related model is the so-called special Hermite operator or \textit{twisted Laplacian} \cite{strichartz_89,thangavelu_98,thangavelu_93,thangavelu_91,wong_05}, which is defined on even-dimensional spaces via an exotic twisted quantization of the harmonic oscillator symbol \cite{degoss_08}. The result is a degenerate elliptic differential operator, which comes with a non-trivial singular space. In \cite{T_tams} we have recently obtained a fairly comprehensive characterization of the modulation space boundedness of the corresponding heat semigroup, with sharp exponential decay over time as in \eqref{eq-intro-expdec}, by exploiting the underlying connections with the Hermite operator via transference principles or the special twisted structure of the propagator. 

We must also mention an important contribution on the microlocal analysis of quadratic operators by Pravda-Starov, Rodino and Wahlberg \cite{pravda_18}. The propagation of singularities for general quadratic operators has been investigated in terms of the so-called \textit{Gabor wave front set} \cite{horm-quad-91,rodino_14,rodino_21}, which is designed to detect the directions of phase space where the Gabor transform of a distribution lacks Schwartz regularity --- see also \cite{carypis,CNR_15,CNT_dispdiag20,wahlberg,white_22_jfa} for generalizations and applications. Perhaps not surprisingly, it turns out that the singular space plays a key role in the context of regularization and propagation of global phase space singularities. To be more precise, the singularities of the initial datum that are located outside the singular space of $q$ are instantaneously smoothed out, while those that are initially inside the singular space remain trapped there and move along the symplectic flow associated with the imaginary part of the symbol, which regulates the metaplectic/dispersive dynamics component of $e^{tq^\w}$ as already seen above. 

\subsection{Main results} To the best of our knowledge, the results just discussed are the only ones currently available in connection with the wave packet analysis of general quadratic operators. The present contribution aims to fill this gap in the literature along two main directions, that are: obtaining a satisfactory counterpart of \eqref{eq-intro-metap-dispdiag} for purely dissipative quadratic semigroups; then extending the $L^2$ analysis of Hitrik and Pravda-Starov in \cite{hitrik} to general modulation spaces. 

In connection with the first goal, our main result is a bound for the magnitude of the Gabor matrix of a quadratic semigroup with non-positive real part and zero singular space, giving simultaneous evidence of the peculiar phase space phenomena (i.e., approximate diagonalization and dissipation, as expected from \eqref{eq-intro-hermite-dec}), in complete analogy to \eqref{eq-intro-metap-dispdiag} in the case of metaplectic operators. 

\begin{theorem}\label{thm-gabmatdec}
	Let $q$ be a complex quadratic form on $\rdd$ with a non positive real part, namely $\re q \le 0$. Let $Q \in \bC^{2d,2d}$ be the companion symmetric matrix and $F=JQ$ the corresponding Hamilton map. Assume that $F$ is diagonalizable and that the singular space \eqref{eq-intro-singsp} associated with $q$ is trivial, that is $S=\{0\}$.
	
	There exists $\mu>0$, depending only on $q$ and $d$, such that for all $t\ge 0$, $N \in \bN$, $g,\gamma \in \cS(\rd)\smo$ and $z,w \in \rdd$, we have
	\begin{equation}\label{eq-gabmatdec}
		\abs{\lan e^{tq^\w} \pi(z)g,\pi(w)\gamma \ran} \le C e^{-t\mu} (1+\abs{w-z})^{-2N},
	\end{equation} for a constant $C>0$ that does not depend on $t,w,z$. 
	
	As a consequence, $e^{tq^\w} \colon \cS(\rd) \to \cS'(\rd)$ extends to a bounded operator on every modulation space $M^p(\rd)$, $1\le p \le \infty$, satisfying
	\begin{equation}\label{eq-intro-modsp-dec}
		\norm{e^{tq^\w}}_{M^p \to M^p} \le C' e^{-t\mu}, 
	\end{equation} for a suitable constant $C'>0$. 
\end{theorem}
To be more precise (cf.\ Proposition \ref{prop-tantF} and Remark \ref{rem-F1eig} below), the assumptions imply that the (possibly repeated) eigenvalues of $F$ occur in opposite non-real pairs $\pm \lambda_j \in \bC$ with $\im \lambda_j \ne 0$, $j = 1,\ldots, d$. If the labels are chosen in such a way that $\im \lambda_j >0$, then the exponent $\mu$ is given by
\begin{equation}\label{eq-intro-mu}
	\mu = \sum_{j=1}^d \im \lambda_j, 
\end{equation} and the resulting exponential decay over time in \eqref{eq-intro-modsp-dec} is actually sharp (cf.\ Remark \ref{rem-sharpdec}).

It is worthwhile to briefly comment on the two assumptions that are made on the Hamilton matrix $F$. Assuming the diagonalizability of $F$ cannot be avoided with the current proof strategy, which relies on a generalization of Williamson's symplectic diagonalization (Proposition \ref{prop-takagi}) to obtain the explicit time decay factor, and on controlling the phase space regularity of the Weyl symbol $\Theta_t$ appearing in the Mehler-H\"ormander formula \eqref{eq-intro-mehler} to obtain fast decay away from the diagonal --- this is indeed a consequence of the fact that $\sup_{(y,\eta) \in \rdq} \abs{(1+\abs{\eta})^{2N}V_G \Theta_t(y,\eta)}< \infty$ for every $N\in \bN$ and $G \in \cS(\rdd)$, which actually entails a weighted modulation space regularity of $\Theta_t$ (cf.\ \eqref{eq-minftys} below). Nevertheless, there is reason to believe that this streamlined approach can be adjusted with some effort to encompass non-diagonalizable Hamiltonian matrices as well, but we preferred to avoid additional technicalities on this occasion (e.g., the occurrence of Jordan blocks \cite{fassb}). Moreover, as evidenced in Section \ref{sec-examples}, this assumption happens to be satisfied by many relevant models. 

On the other hand, assuming triviality of the singular space might seem a more serious restriction --- see also \cite{hitrik_18,hitrik_08,ottobre,pravda_11,white_22} for additional results on quadratic operators in presence of this constraint. In fact, although the proof leaves room for some slight relaxation of this assumption (e.g., it suffices for $F$ to have non-real eigenvalues), the setting of Theorem \ref{thm-gabmatdec} virtually captures the standard behavior of the dissipative component of every quadratic semigroup, and there is a way to always reduce to this case provided that the effect of the complementary dispersive components is duly taken into account. To make this heuristics more precise, we need to recall a clever decomposition of the quadratic form $q$ introduced in \cite{hitrik} --- see also Proposition \ref{prop-split} below. In short, for every general quadratic form $q$ with $\re q\le 0$ and symplectic singular space $S$, there exists a symplectic transformation $\chi \colon \rdd \to \rdd$ such that
\begin{equation}
		(q\circ \chi)(x,\xi) = q_1(x',\xi') + iq_2(x'',\xi''), \qquad \chi(x,\xi) = (x',x'';\xi',\xi''), 
\end{equation} where $(x',\xi') \in \rnnp$ and $(x'',\xi'')\in \rnns$ are symplectic coordinates on $\Sbots$ and $S$ respectively --- hence the decomposition $T^*\rd = \Sbots \oplus^{\bots} S$ holds. In particular, $q_1 = \restr{q}{\Sbots}$ is a complex quadratic form on $\rnnp$, with $\re q_1 \le 0$ and trivial singular space $S_1 = \{0\}$. On the other hand, $q_2 = \im \restr{q}{S}$ is a real quadratic form on $\rnns$ that thus encodes the self-adjoint features of the quadratic semigroup $e^{tq^\w}$. 

This result, in accordance with those in \cite{pravda_18} already discussed, suggests that the phase space analysis should be carried separately on the singular space and its symplectic complement, in order to disentangle the dissipative and dispersive effects of the quadratic semigroup. After that, by combining the results in Theorem \ref{thm-gabmatdec} and \eqref{eq-intro-metap-dispdiag}, we obtain a pointwise bound for the Gabor matrix of $e^{tq^\w}$ (cf.\ \eqref{eq-gabmat-mix} below), which in turn allows us to extend the $L^2$ analysis performed in \cite{hitrik} to the entire range of modulation spaces, hence achieving a fairly complete picture of the phase space regularity of a general quadratic semigroup. 

\begin{theorem}\label{thm-main-mod}
	Let $q$ be a complex quadratic form on $\rdd$ with a non positive real part, namely $\re q \le 0$. Let $Q \in \bC^{2d,2d}$ be the companion symmetric matrix and $F=JQ$ the corresponding Hamilton map.
	
	Assume that the singular space $S$ defined in \eqref{eq-intro-singsp} is a symplectic subspace of $T^*\rd$, with $\dim S = 2n$, $0 \le n < d$, and that $\restr{F}{\Sbots}$ is diagonalizable. 
	
	For all $t \ge 0$, the operator $e^{tq^\w} \colon \cS(\rd) \to \cS'(\rd)$ extends to a bounded operator on $M^p(\rd)$, $1 \le p \le \infty$. In particular, there exist $C>0$ independent of $t$ such that
	\begin{equation}
		\|e^{tq^\w}\|_{M^p \to M^p} \le C e^{-\mu t} (\sigma_1(t) \cdots \sigma_{n}(t))^{ \abs{\frac12-\frac1p}}, 
	\end{equation} where $\mu=\mu(d-n)$ is the quantity defined in \eqref{eq-intro-mu} associated with $\restr{q}{\Sbots}$, and   $\sigma_1(t) \ge \ldots \ge \sigma_n(t) \ge 1$ are the $n$ largest singular values of the symplectic flow $e^{2t\im(\restr{F}{S})} \in \Sp(n,\bR)$. 
\end{theorem}

Note that $S= \rdd$ if and only if $\re q = 0$ identically \cite[Theorem 1.2.3]{hitrik}. In that case $e^{tq^\w} = e^{it(\im q)^\w}$ is a metaplectic operator and the previous result holds with no exponential decay (i.e., $\mu(d)=0$), so that \eqref{eq-intro-metbound} is recovered. Moreover, the assumption of diagonalizability of the Hamilton matrix inherited from Theorem \ref{thm-gabmatdec} is actually required to hold only for the restriction $\restr{F}{\Sbots}$ to the symplectic orthogonal complement of the singular space. 

We also stress that the assumption on the symplectic structure of $S$ cannot be relaxed in general --- simple counterexamples are given by the semigroups generated by $q(x,\xi)=-x^2$ or $q(x,\xi)=-\xi^2$ (standard heat equation), whose Hamiltonian matrices are nilpotent and the singular spaces are Lagrangian planes (i.e., satisfying $\Sbots = S$). As a matter of fact, suitable modifications of the proof of Theorem \ref{thm-gabmatdec}, or standard arguments from pointwise and Fourier multipliers theory on modulation spaces, show that $e^{tq^\w}$ is bounded $M^p(\rd)\to M^p(\rd)$ for every $1 \le p \le \infty$, the operator norm being uniformly bounded with respect to $t$. 

A brief outline of the manuscript follows. In Section \ref{sec-prep} we review some preliminaries for later use, especially including the generalized Williamson theorem (in fact, a symplectic variant of the Autonne-Takagi factorization) in Proposition \ref{prop-takagi}. Section \ref{sec-proofs} is then devoted to the proof of the main results, namely Theorems \ref{thm-gabmatdec} and \ref{thm-main-mod}. Finally, some applications are discussed in Section \ref{sec-examples}.

\section{Preparation}\label{sec-prep} 
\subsection{Notation} We agree that $\bN$ denotes the set of non-negative integers. The integer part of $x \in \bR$ is denoted by $\floor{x}$. 

The inner product of $x,y \in \rd$ is denoted by $x\cdot y$, and we write $x^2$ in place of $x\cdot x$. 

The inner product between $f,g \in L^2(\rd)$ is defined by
\begin{equation}    \lan f,g \ran = \ird f(y) \overline{g(y)} \dd{y}, 
\end{equation} and extends to the duality pairing between a temperate distribution $f \in \cS'(\rd)$ and a function $g \in \cS(\rd)$ in the Schwartz class upon agreeing that $\lan f,g \ran = f(\overline{g})$.

Given $X,Y \in \bR$, we write $X \lesssim_\lambda Y$ as a shorthand for the statement $X \le C Y$, where $C>0$ is a constant that does not depend on $X,Y$ but may depend on the parameter $\lambda$. We also write $X \asymp_\lambda Y$ when both $X \lesssim_\lambda Y$ and $Y \lesssim_\lambda X$ hold. 

Recall that the inhomogeneous magnitude of $x \in \rd$ is denoted by $\lan x \ran \coloneqq (1+|x|^2)^{1/2}$. We also introduce the functions $v_s(x) \coloneqq (1+|x|)^s$, $s \in \bR$. Note that $\lan x \ran^s \asymp v_s(x)$ for every $s \in \bR$. We recall for later use the subconvolutivity property from \cite[Lemma 11.1.1]{gro_book}: 
	\begin{equation}\label{eq-subconv}
	v_{-s} * v_{-s}(x) \lesssim_s v_{-s}(x), \qquad \forall \, s>d.
\end{equation}

	The tensor product $f_1 \otimes f_2$ of two temperate distributions $f_1,f_2 \in \cS'(\rd)$ is the unique element of $\cS'(\rdd)$ satisfying the identity
\begin{equation}
	\lan f_1\otimes f_2, g_1 \otimes g_2 \ran = \lan f_1,g_1 \ran \lan f_2,g_2 \ran, \quad g_1,g_2 \in \cS(\rd),
\end{equation} where $g_1\otimes g_2 \in \cS(\rdd)$ is defined as usual by $g_1 \otimes g_2(x,y) = g_1(x)g_2(y)$, $(x,y) \in \rdd$. 

The tensor product of two operators $T_1,T_2 \colon \cS'(\rd) \to \cS'(\rd)$ is the unique operator $T_1 \otimes T_2 \colon \cS'(\rdd)\to \cS'(\rdd)$ such that
\begin{equation}
	(T_1\otimes T_2)(u_1 \otimes u_2) = T_1u_1 \otimes T_2u_2, \quad \forall u_1,u_2 \in \cS'(\rd). 
\end{equation}

The direct sum $A \oplus B \in \bC^{n,n}$ of two matrices $A,B \in \bC^{n,n}$ is the block diagonal matrix defined as follows:
\begin{equation}
	A\oplus B \coloneqq \begin{bmatrix} A & O \\ O & B \end{bmatrix}.
\end{equation}
We also write $\diag(A,B)$ in place of $A\oplus B$, thus extending the standard notation $\diag(\delta_1,\ldots,\delta_n) \in \bC^{n,n}$ for the diagonal matrix with entries $\delta_1,\ldots,\delta_n \in \bC$. 		

\subsection{Gabor analysis}
	The Fourier transform is normalized here as follows:
\begin{equation}
	\cF f(\xi)= \wh{f}(\xi) \coloneqq (2\pi)^{-d/2} \ird e^{-i\xi \cdot x}f(x) \dd{x}.
\end{equation} 

The Gabor transform of $f \in \cS'(\rd)$ with respect to the window $g \in \cS(\rd)\smo$ is defined, as already anticipated in the Introduction, by
\begin{equation}
	V_g f(z) \coloneqq \lan f,\pi(z)g \ran = \cF (f \overline{g(\cdot - x)})(\xi), \qquad z=(x,\xi)\in \rdd,
\end{equation} where $\pi(z)$ denotes the phase space shift along $z=(x,\xi) \in T^*\rd$:
\begin{equation}
	\pi(z)g(y) = (2\pi)^{-d/2} e^{i\xi \cdot y} g(y-x), \qquad y \in \rd.
\end{equation}

We collect below some basic properties of the Gabor transform --- see \cite{gro_book} for proofs and additional details. 

\begin{lemma} \label{lem-stft} For every $f,g,h,\gamma \in \cS(\rd)\smo$ and $u,w,z \in \rdd$ we have: 
	\begin{itemize}
		\item Schwartz regularity: for every $N \ge 0$, 
		\begin{equation}
			\abs{V_g f (z)} \lesssim_N \lan z \ran^{-N}. 
		\end{equation}
		
		\item Covariance formula:
	\begin{equation}
			\abs{V_{\pi(z)g}{\pi(w)\gamma}(u)} = \abs{V_g \gamma(u-w+z)}. 
		\end{equation}		
		
		\item Orthogonality relation:
	\begin{equation}
		\lan V_g f, V_\gamma h \ran_{L^2(\rdd)} = \lan f,h \ran_{L^2(\rd)} \overline{\lan g,\gamma \ran_{L^2(\rd)}}. 
	\end{equation}
	
		\item Inversion formula: if $f \in \cS'(\rd)$ and $\lan \g,g \ran \ne 0$, then
		\begin{equation}\label{eq-inv-stft}
			f = \frac{1}{\lan \gamma, g \ran} \irdd V_g f(z) \pi(z) \gamma \dd{z},
		\end{equation} where the identity is meant in the sense of distributions. 
	\end{itemize}
	\end{lemma}
	
	The phase space summability of the Gabor transform characterizes the so-called modulation spaces. Given $1 \le p,q \le \infty$ and $g \in \cS(\rd)$, the modulation space $M^{p,q}(\rd)$ is the Banach space of distributions $f \in \cS'(\rd)$ for which the norm
	\begin{equation}\label{eq-modsp-norm}
		\norm{f}_{M^{p,q}} \coloneqq \Big( \ird \Big( \ird \abs{V_g f(x,\xi)}^p \dd{x} \Big)^{q/p} \dd{\xi} \Big)^{1/q} 
	\end{equation} is finite, with obvious modifications in the case where $p=\infty$ or $q=\infty$. We write $M^p(\rd)$ if $q=p$. The reader is addressed to \cite{CR_book,gro_book} for additional details. 
	
\subsubsection{Weyl pseudo-differential operators}
We already defined the Weyl quantization in the Introduction, cf.\ \eqref{eq-intro-weylq}. More generally, the pseudo-differential operator $\sigma^\w \colon \cS(\rd) \to \cS'(\rd)$ with Weyl symbol $\sigma \in \cS'(\rdd)$ can be defined in the sense of distributions by requiring that
\begin{equation}
	\lan \sigma^\w f,g \ran = \lan \sigma, W(g,f) \ran, \qquad f,g \in \cS(\rd),
\end{equation} where $W(g,f)$ is a sesquilinear phase space representation known as the \textit{(cross-)Wigner transform} of $f,g$:
	\begin{equation}
	W(g,f)(z) \coloneqq (2\pi)^{-d} \ird e^{-i\xi \cdot y} g(x+y/2) \overline{f(x-y/2)} \dd{y}, \qquad z=(x,\xi) \in \rdd.
\end{equation} 

A straightforward computation shows that the Wigner transform satisfies the following covariance property: 
\begin{equation}
	\abs{W(\pi(u)f,\pi(v)g)(z)} = \abs{\pi\Big(\frac{u+v}{2},J(u-v)\Big) W(f,g)(z)}, \qquad u,v,z \in \rdd.
\end{equation} This relation can be used to unveil the connection between the Gabor matrix of a Weyl operator and the Gabor transform of its symbol: if $g,\gamma \in \cS(\rd)$ and $z,w \in \rdd$, 
\begin{align}\label{eq-gabmat-sym}
	\abs{\lan \sigma^\w \pi(z)g,\pi(w)\gamma \ran} & = \abs{\lan \sigma, W(\pi(w)\gamma,\pi(z)g) \ran} \\
	& = \abs{\lan \sigma, \pi\Big(\frac{w+z}{2},J(w-z)\Big) W(\gamma,g)\ran} \\
	& = \abs{ V_{G} \sigma \Big(\frac{w+z}{2}, J(w-z)\Big)}, 
\end{align} where $G=W(\g,g) \in \cS(\rdd)$. It is therefore clear that to obtain off-diagonal decay for the Gabor matrix of $\sigma^\w$ it suffices that the Gabor transform of $\sigma$ decays with respect to the frequency variable, say
\begin{equation}
	\sup_{u \in \rdd} \abs{V_G \sigma(u,v)} \lesssim \lan v \ran^{-s}, 
\end{equation} for some $s>0$. Equivalently, this condition is met if $\sigma$ belongs to a weighted $M^\infty$ space, that is
\begin{equation}\label{eq-minftys}
	\norm{\sigma}_{M^\infty_{0,s}} \coloneqq \sup_{(u,v) \in \rdq} \lan v \ran^s\abs{V_G \sigma(u,v)} < \infty. 
\end{equation}
The interested reader may wish to consult \cite{CR_book,gro_book,NT_book} for further information on this approach --- see also \cite{T_JDE,wahl_vec} for vector-valued models. 

\subsubsection{Metaplectic operators} \label{sec-metap} 

The group $\Spdc$ of complex symplectic matrices consists of $M \in \bC^{2d,2d}$ satisfying 
\begin{equation}
	M^\top J M = J,
\end{equation} where $J$ is the symplectic matrix defined in \eqref{eq-intro-defJ}. The same condition characterizes real symplectic matrices $M \in \bR^{2d,2d}$ in $\Spdr$. For the sake of readability we omit the dependence on the dimension for $J$, or the identity/null matrices, when the context is clear. 

Recall from \cite{degoss_11_book,folland} that the metaplectic representation is a unitary representation on $L^2(\rd)$ of the two-fold covering of the real symplectic group $\Spdr$. Every symplectic matrix $S \in \Spdr$ thus associates, up to the sign, with a unitary operator $\mu(S)$. The metaplectic representation can be equivalently characterized via intertwining with phase-space shift. Given the current normalization, we have indeed
\begin{equation}\label{eq-metap-intertw}
	\pi(S z)=c_S(z) \mu(S) \pi(z) \mu(S)^{-1},\quad z \in \rdd, 
\end{equation} where $c_S(z) \in \bC$ is a suitable phase factor, viz., $\abs{c_S(z)}=1$. 

We also recall that a distinctive property of the Weyl quantization is the so-called \textit{symplectic covariance}: for every $S \in \Spdr$ and $\sigma \in \cS'(\rdd)$, we have
\begin{equation}\label{eq-sympcov}
	(\sigma \circ S)^\w = \mu(S)^{-1} \sigma^\w \mu(S).
\end{equation} 

We review some results proved in \cite{CNT_dispdiag20}, to which the reader is addressed for further details. Recall that the singular values of a symplectic matrix $S \in \Spdr$ occur in $d$ pairs of reciprocal positive real numbers. Consider the labeling such that $\sigma_1 \ge \ldots \ge \sigma_d \ge 1 \ge \sigma_d^{-1} \ldots \ge \sigma_1^{-1}$, and set $\Sigma=\diag(\sigma_1,\ldots,\sigma_d)$. An \textit{Euler decomposition} of $S \in \Sp$ is nothing but a symplectic version of the standard SVD factorization. To be precise, there exist (non-unique) orthogonal and symplectic matrices $U,V \in \bR^{2d,2d}$ such that $S=U^\top (\Sigma \oplus \Sigma^{-1}) V$. The triple $(U,V,\Sigma)$ thus identifies an Euler decomposition of $S$. 

\begin{theorem}\label{thm-gabmatdisp}
	Let $h \colon \rnn \to \bR$ be a real quadratic form, $H \in \bR^{2n,2n}$ being the companion symmetric matrix. For every $t \in \bR$, the corresponding Schr\"odinger propagator $e^{ith^\w}$ coincides (up to the sign) with the metaplectic operator $\mu(H_t)$, where $H_t = e^{2JHt}\in \Spdr$.
	
	Moreover, for every $g, \gamma \in \cS(\rd)\smo$ and $N >0$ there exists $C>0$ such that, for every Euler decomposition $(U_t,V_t,\Sigma_t)$ of $H_t$, we have	 
	\begin{equation}\label{eq-dispdiag}
		\abs{\lan e^{ith^\w} \pi(z)g, \pi(w)\gamma \ran} \le C (\det \Sigma_t)^{-1/2} (1+|D_t U_t(w-H_t z)|)^{-N}, \quad z,w \in \rdd,  
	\end{equation} where $D_t = \Sigma_t^{-1} \oplus I$. 
\end{theorem} 

\subsection{Some results on symplectic diagonalization}
The following result plays a key role in the proof of the main result, and is a symplectic variant of the standard Autonne-Takagi factorization of a complex symmetric matrix \cite{horn}. In particular, sufficient conditions are given for the latter to be symplectically congruent to a diagonal matrix. It can be also viewed as a generalization of the symplectic diagonalization of real, positive definite quadratic forms by Williamson \cite[Theorem 93]{degoss_11_book}, which is indeed recaptured as a special case (although under slightly stronger assumptions). 

\begin{proposition}\label{prop-takagi}
	Consider a complex symmetric matrix $Q \in \bC^{2d,2d}$, $Q^\top = Q$, and the corresponding Hamilton matrix $F=JQ$. 
	\begin{itemize}
		\item The (possibly repeated) eigenvalues of $F$ occur in opposite pairs $\pm \lambda_j$ with $\lambda_j \in \bC$, $j=1, \ldots, d$.  
		\item If $F$ is diagonalizable then there exists a symplectic matrix $P \in \Spdc$ such that 
		\begin{equation}
			P^\top Q P =D, \qquad D=\begin{bmatrix}
				i\Lambda & O \\ O & i\Lambda
			\end{bmatrix}, \qquad \Lambda=\diag(\lambda_1,\ldots, \lambda_d). 
		\end{equation}
		Such a factorization is unique up to relabeling the entries of $\Lambda$ or replacing $\Lambda$ with $-\Lambda$. 
	\end{itemize}
\end{proposition} The proof of the first claim is a trivial consequence of the similarity $F^\top = -J^{-1}FJ$. A proof of the second part, up to minor notational differences, can be found in \cite[Theorem 21]{delacruz}. In fact, the reader can find there a full characterization of symplectic diagonal congruence, showing that $Q \in \bC^{2d,2d}$ is symplectically congruent to a diagonal matrix if and only if $Q$ is symmetric and $F^2$ is diagonalizable. A more constructive yet elementary proof can be given along the lines of \cite[Theorem 1]{kim}, provided that one conveniently incorporates the fact that the matrices $\diag(\lambda,-\lambda)$ and $iJ\diag(\lambda, \lambda)$ are similar for every $\lambda \in \bC$. 
	
%
	
We now prove how the previous symplectic factorization can be leveraged in connection with the generalized Mehler formula \eqref{eq-intro-mehler}.	
	\begin{proposition} \label{prop-tantF}
		Consider a complex symmetric matrix $Q \in \bC^{2d,2d}$, hence satisfying $Q^\top = Q$, and the corresponding Hamilton matrix $F=JQ$.  
		
		\begin{enumerate}
			\item If $F$ has no real eigenvalue, the matrix $\cos(tF)$ is invertible for every $t \ge 0$ and there exists $\mu >0$, depending only on $Q$ and $d$, such that 
			\begin{equation}
				\abs{\det^{-1/2}(\cos(tF))} \lesssim_{d,Q} e^{-t\mu }. 
			\end{equation}
			\item If $F$ is diagonalizable, then there exists $P \in \Spdc$ such that
			\begin{equation}
				\tan(tF) = J P^{-\top} \begin{bmatrix}
					i\tan(t\Lambda) & O \\ O & i\tan(t\Lambda)
				\end{bmatrix} P^{-1}. 
			\end{equation}
		\end{enumerate}
	\end{proposition}
	
	\begin{proof}
		Let $\{\pm \lambda_j : j = 1,\ldots,d\}$ be the set of eigenvalues of $F$. Without loss of generality, labels are chosen in such a way that $\im(\lambda_j)>0$. 
		
		We first note that if no eigenvalue of $F$ is real then, for every $t\ge 0$, \begin{equation}
			\det(\cos(tF)) = \prod_{j=1}^d \cos(t\lambda_j)^2 \ne 0.
		\end{equation} To be precise, after setting $\a_j=\re(\lambda_j)$ and $\b_j=\im(\lambda_j)$ we have
		\begin{equation}
			\cos(t\lambda_j) = \cos(t\alpha_j)\cosh(t\beta_j)-i\sin(t\alpha_j)\sinh(t\beta_j),
		\end{equation} hence
		\begin{align}
			\abs{\cos(t\lambda_j)} & = (\cos(t\alpha_j)^2\cosh(t\beta_j)^2+\sin(t\alpha_j)^2\sinh(t\beta_j)^2)^{1/2} \\
			& = 2^{-1/2} ( \cos(2t\alpha_j) + \cosh(2t\beta_j))^{1/2} \\ 
			& = 2^{-1/2} \cosh(2t\beta_j)^{1/2} \Big( 1+ \frac{\cos(2t\alpha_j)}{\cosh(2t\beta_j)} \Big)^{1/2}. 
		\end{align}
		It is easy to see that for every $\alpha_j,\beta_j$ there exists $C_j>0$ such that, for every $t\ge 0$, 
		\begin{equation}
			\abs{\cos(t\lambda_j)} > 2^{-1/2}C_j^{1/2} \cosh(2t\beta_j)^{1/2} \ge 2^{-1}C_j^{1/2} e^{t\beta_{j}}.
		\end{equation}
		Therefore,
		\begin{equation}
			\abs{\det^{-1/2}(\cos(tF))}  = \prod_{j=1}^d \abs{\cos(t\lambda_j)}^{-1} < \prod_{j=1}^d  2C_j^{-1/2}e^{-t\beta_j} < C_Q^n e^{-t\sum_{j=1}^n\beta_j},
		\end{equation} where we set $ C_Q \coloneqq 2\max \{C_j^{-1/2} : j=1,\ldots,d\}$. The claim thus follows with $\mu = \sum_{j=1}^d\beta_j$. 
		
		Let us now compute $\tan(tF)$. Recalling the Takagi symplectic factorization of $Q$ from Proposition \ref{prop-takagi}, we have
		\begin{equation}
			F= JQ = JP^{-\top} D P^{-1}, \qquad D=\diag(i\Lambda,i\Lambda), 
		\end{equation} hence for every $k\in \bN$ we recursively have
		\begin{equation}
			F^k = J P^{-\top} D (J D)^{k-1} P^{-1} = J P^{-\top} (DJ)^k J^{-1} P^{-1}.
		\end{equation} 
		Straightforward computations show that 
		\begin{equation}
			(DJ)^k = \begin{cases}
				\diag(\Lambda^k,\Lambda^k) & (k \text{ even}) \\  \diag(i\Lambda^k,i\Lambda^k) J & (k \text{ odd}).			
			\end{cases}
		\end{equation}
		We infer that
		\begin{align}
			\sin(tF) & = \sum_{k=0}^\infty \frac{(-1)^k}{(2k+1)!} t^{2k+1} F^{2k+1} \\
			& = JP^{-\top} \Big( \sum_{k=0}^\infty \frac{(-1)^k}{(2k+1)!} t^{2k+1} (DJ)^{2k+1} \Big) J^{-1} P^{-1} \\ 
			& = J P^{-\top} \diag(i \sin(t\Lambda),i\sin(t\Lambda)) P^{-1}. 
		\end{align} 
		In a similar fashion we obtain
		\begin{align}
			\cos(tF) & = \sum_{k=0}^\infty \frac{(-1)^k}{(2k)!} t^{2k} F^{2k} \\
			& = JP^{-\top} \Big( \sum_{k=0}^\infty \frac{(-1)^k}{(2k)!} t^{2k} (DJ)^{2k} \Big) J^{-1} P^{-1} \\ 
			& = J P^{-\top} \diag(\cos(t\Lambda),\cos(t\Lambda)) J^{-1} P^{-1}, 
		\end{align} 
		hence, using the fact that $P$ is symplectic,
		\begin{align}
			(\cos(tF))^{-1} & = P J \diag(\cos(t\Lambda)^{-1},\cos(t\Lambda)^{-1}) P^{\top}J^{-1} \\
			& = P J \diag(\cos(t\Lambda)^{-1},\cos(t\Lambda)^{-1}) J^{-1} P^{-1} \\
			& = P \diag(\cos(t\Lambda)^{-1},\cos(t\Lambda)^{-1}) P^{-1}.
		\end{align}
		To conclude,
		\begin{align}
			\tan(tF) & = (\sin(tF))(\cos(tF))^{-1} \\
			& = J P^{-\top} \diag(i\tan(t\Lambda),i\tan(t\Lambda))P^{-1} \\
			& = (\cos(tF))^{-1}(\sin(tF)),  
		\end{align} where in the last step we used that $JP^{-\top} = PJ$, since $P^\top$ is a  symplectic matrix. 
	\end{proof}

\section{Proofs of the main results}\label{sec-proofs} 

\subsection{Proof of Theorem \ref{thm-gabmatdec}} We provide now a proof of Theorem \ref{thm-gabmatdec}. First, we make a crucial remark which enables resorting to Proposition \ref{prop-takagi}. 

\begin{remark}\label{rem-F1eig}
	Consider the setting of Theorem \ref{thm-gabmatdec}. The assumption on the triviality of the singular space, that is $S=\{0\}$, implies that the Hamilton matrix $F$ cannot have real eigenvalues. Indeed, if $\lambda \in \bR$ were an eigenvalue of $F$, the space $A_\lambda \coloneqq (\ker(F-\lambda I) \oplus \ker(F+\lambda I)) \cap \rdd$ would be non-empty and $\ne \{0\}$, but it is known that the inclusion $A_\lambda \subset S$ holds for every $\lambda \in \bR$ --- cf.\ \cite[Page 812]{hitrik}.
\end{remark}
	
\begin{proof}[Proof of Theorem \ref{thm-gabmatdec}]
Set $\theta_t = \det^{1/2} (\cos(tF)) \Theta_t$ and $G =W(\gamma,g) \in \cS(\rdd)$. In light of the identity \eqref{eq-gabmat-sym} and Proposition \ref{prop-tantF}, we have 
	\begin{align}
		\abs{\lan e^{tq^\w} \pi(z)g,\pi(w)\gamma \ran} & 
		= \abs{\lan \Theta_t^\w \pi(z)g,\pi(w)\gamma \ran} \\
		& = \abs{\det^{-1/2}(\cos(tF))} \abs{V_G \theta_t \Big( \frac{w+z}{2}, J(w-z) \Big) } \\
		& \lesssim_{d,q} e^{-t\mu} \abs{V_G \theta_t \Big( \frac{w+z}{2}, J(w-z) \Big) }.
	\end{align} It is thus enough to prove that for every integer $N \ge 0$ we have
	\begin{equation}
		|V_G \theta_t(u,v)| \lesssim_{d,N,q} \lan v \ran^{-2N}, \qquad (u,v) \in \rdq.
	\end{equation} 
	In light of Proposition \ref{prop-tantF}, after setting $\wt \theta_t(y) =  e^{y \cdot  \diag(i\tan(t\Lambda),i\tan(t\Lambda))y}$ we have
	\begin{align}
		\abs{V_G \theta_t (u,v)} & \asymp_d \abs{ \int_{\rdd} e^{-i v \cdot y} \exp(Jy \cdot \tan(tF)y) \,  \overline{G(y-u)} \dd{y}} \\
		& = \abs{\int_{\rdd} e^{-i v \cdot y} \exp(Jy \cdot JP^{-\top} \diag(i\tan(t\Lambda),i\tan(t\Lambda))P^{-1}y) \,  \overline{G(y-u)} \dd{y}} \\
		& = \abs{\int_{\rdd} e^{-i P^{\top} v \cdot y} \wt{\theta_t}(y) \,  \overline{G(Py-u)} \dd{y}}.
	\end{align}
	Assume $N>0$ first. Using the identity $(1-\Delta_y)^N e^{- i P^{\top}v \cdot y} = \lan P^{\top}v \ran^{2N}	e^{- i P^{\top}v \cdot y}$, integration by parts yields
	 \begin{align}
	 	\abs{V_G \theta_t (u,v)}
	 	& \lesssim_d \lan v \ran^{-2N} \int_{\rdd} \abs{(1-\Delta_y)^N e^{y \cdot  \diag(i\tan(t\Lambda),i\tan(t\Lambda))y} \,  \overline{G(Py-u)} \dd{y}}. 
	 \end{align} 
	We claim that the remaining integral is finite and uniformly bounded with respect to $t$. Indeed, we can write
	 \begin{align}
	 	\abs{(1-\Delta_y)^N \big( \wt{\theta_t}(y) \overline{G(Py-u)} \big)} & \le \sum_{\abs{\a}+\abs{\b}\le 2N} \abs{c_{\a,\b}} \abs{\partial^\a \wt{\theta_t}(y)} \abs{\partial^\b G(Py-u)},
	 \end{align} for suitable coefficients $c_{\a,\b} \in \bC$. Note that $G\circ P$ is a Schwartz function and $\int_{\rdd} \lan Py-u \ran^{-m} \dd{y} = \| v_{-m}\|_{L^1} < \infty$ for every $m>2d$, so it suffices to prove that every term $\abs{\partial^\a \wt{\theta_t}}$ is bounded, uniformly with respect to $t$.
	 
	 To this aim, setting $y=(\eta,\xi) \in \rdd$ and $\lambda_j = \a_j + i \b_j$ for $\a_j,\b_j \in \bR$ with $\b_j>0$ (cf.\ Remark \ref{rem-F1eig}) yields the explicit representation 
	 \begin{align}
	 	\wt{\theta_t}(y) & = \exp\Big( \sum_{j=1}^d i \tan(t\lambda_j) (\eta_j^2+\xi_j^2) \Big) \\
	 	& = \exp\Big( \sum_{j=1}^d  (-\rho_j(t)+ i \iota_j(t))(\eta_j^2+\xi_j^2) \Big),
	 \end{align}
	 where  
	 \begin{equation}
	 	\rho_j(t) \coloneqq \frac{\sinh(2\beta_j t)}{\cos(2\alpha_j t)+\cosh(2\beta_j t)}, \qquad \iota_j(t) \coloneqq  \frac{\sin(2\alpha_j t)}{\cos(2\alpha_j t)+\cosh(2\beta_j t)}.
	 \end{equation}  
	  A simple tensorization argument shows that it is enough to fix $j\in \{1,\ldots,d\}$ and focus on the function 
	  \begin{equation}
	  	f_j(x) =e^{(-\rho_j(t)+i\iota_j(t)) x^2}, \qquad x \in \bR,
	  \end{equation} proving that $\abs{f_j^{(n)}}$ is uniformly bounded with respect to $t>0$ for every $n\in \bN$. 
	  
	  Set $\gamma_j(t)=-\rho_j(t)+i \iota_j(t)$. It is not difficult to show by induction that \begin{equation}
	  	f_j^{(n)}(x)\asymp_n  \gamma_j(t)^{n-\floor{n/2}} x^{n-2\floor{n/2}} P_{\floor{n/2}}(\gamma_j(t)x^2) e^{\gamma_j(t)x^2},
	  \end{equation} where $P_{\floor{n/2}}$ is a polynomial of degree $\floor{n/2}$ with non-negative integer coefficients. As a result, we have
	  \begin{equation}
	  	\abs{f_j^{(n)}(x)} \lesssim_n \abs{\gamma_j(t)}^{n-\floor{n/2}} \abs{x}^{n-2\floor{n/2}} \Big( \sum_{k=0}^{\floor{n/2}}  \abs{\gamma_j(t)}^k x^{2k} \Big) e^{-\rho_j(t)x^2}. 
	  \end{equation} Recall that there exists $C_j>0$ such that, for every $t\ge 0$,
	  \begin{equation}
	 1+\frac{\cos(2\a_j t)}{\cosh(2\b_jt)} > C_j. 
	  \end{equation}
	  Therefore, we have $0 < \rho_j(t) < C_j^{-1}$ for every $t >0$ and $\rho_j(t)=0$ if and only if $t=0$. Similarly, we have $0 \le \abs{\iota_j(t)} < C_j^{-1}$ for all $t\ge 0$, hence $\abs{\gamma_j(t)} \lesssim_j 1$.
	  
	 It is clear that there exists $t_0>0$ sufficiently small such that $\abs{\iota_j(t)} \lesssim_j \rho_j(t)$, hence $\abs{\gamma_j(t)} \lesssim_j \rho_j(t)$, if $0<t<t_0$. Since the function $g_m(x)=x^m e^{-\rho_j(t)x^2}$, $m\ge 1$, achieves its maximum value at $x \asymp_m \rho_j(t)^{-1/2}>0$, if $0<t<t_0$ we conclude that 
	  \begin{equation}
	  	\abs{f_j^{(n)}(x)} \lesssim_{n,j} \rho_j(t)^{n-\floor{n/2}} \abs{x}^{n-2\floor{n/2}} \Big( \sum_{k=0}^{\floor{n/2}} \rho_j(t)^k x^{2k} \Big) e^{-\rho_j(t)x^2} \lesssim \rho_j(t)^{n/2} \lesssim_j 1. 
	  \end{equation}	  
	  On the other hand, if $t\ge t_0 >0$ then $\rho_j(t) \ge \tanh(2\beta_j t_0)/2>0$, hence $\rho_j(t)^{-1/2} \lesssim_{j,t_0} 1$. The bound $\abs{\gamma_j (t)} \lesssim_j 1$ and the same maximization argument as above now yield
	  \begin{equation}
	  	\abs{f_j^{(n)}(x)} \lesssim_{n,j}  \abs{x}^{n-2\floor{n/2}} \Big( \sum_{k=0}^{\floor{n/2}} x^{2k}  e^{-\rho_j(t)x^2}\Big) \lesssim_{n,j} 1. 
	  \end{equation} 
	  	
	To conclude, note that the claim trivially holds in the case where $N=0$, since the explicit representation of  $\wt{\theta_t}$ obtained above shows that the latter is in fact a bounded function on $\rdd$. 
	
	Continuity on modulation spaces then follows at once from \eqref{eq-intro-lens} via Young's inequality. 
\end{proof}

\begin{remark}
	Note that \eqref{eq-gabmatdec} and Young's inequality actually imply continuity of $e^{tq^\w}$ on every modulation space $M^{p,q}(\rd)$ with $1 \le p,q \le \infty$. In particular, we have
	\begin{equation}
		\norm{e^{tq^\w}}_{M^{p,q}\to M^{p,q}} \lesssim e^{-t \mu}. 
	\end{equation}
	Further generalizations are possible, including weighted versions as well as continuity results in the quasi-Banach regime where $0<p,q<1$ --- the latter can be obtained by means of a discretization argument along a suitable Gabor frame, see \cite{bast_cn,cg,CR_book,galperin,toft} in this connection. The related details are left to the interested reader. 
	
	We also highlight that in \cite{hitrik,pravda_18} it is proved that the semigroup $e^{tq^\w}$ is actually infinitely regularizing, that is $e^{tq^\w}f \in \cS(\rd)$ for every $f \in \cS'(\rd)$, but this fact cannot be apparently recovered from the results above in general.
\end{remark}
	
\begin{remark}\label{rem-sharpdec}
	The time dependence in \eqref{eq-intro-modsp-dec} is sharp. Indeed, recall from \cite[Theorem 1.2.2]{hitrik} that the spectrum of $q^\w$ has the form
	\begin{equation}
		\sigma(q^\w) = \Bigg\{  \sum_{\substack{\lambda \in \sigma(F) \\ \re(-i\lambda)<0}} (r_\lambda+2k_\lambda) (-i\lambda) : k_\lambda \in \bN \Bigg\},
	\end{equation}
	where $r_\lambda$ is the dimension of the space of generalized eigenvectors in $\bC^{2d}$ associated with the eigenvalue $\lambda$, which coincides with the algebraic multiplicity of $\lambda$ in the case where $F$ is diagonalizable. The first eigenvalue in the bottom of the spectrum is thus given by
	\begin{equation}
		\mu_0 = \sum_{\substack{\lambda \in \sigma(F) \\ \re(-i\lambda)<0}} -i\lambda r_\lambda. 
	\end{equation} It is known that it has multiplicity one and its eigenspace is spanned by a ground state $\psi_0 \in \cS(\rd)$ of exponential type (\cite[Theorem 2.1]{ottobre}). Compactness of $e^{tq^\w}$ for every $t>0$ \cite[Proposition 3.1.1]{hitrik} implies the spectral characterization $\sigma(e^{tq^\w}) = \{0\}\cup \{e^{t\gamma} : \gamma \in \sigma(q^\w)\}$, from which we infer
	\begin{equation}
		\|e^{tq^\w}\psi_0 \|_{M^p} 
		= e^{-t\mu} \|\psi_0\|_{M^p},
	\end{equation} where, since the eigenvalue list $\lambda_1,\ldots, \lambda_d$ takes into account possible repetitions,
	\begin{equation}
		\mu = -\re(\mu_0) = -\sum_{\substack{\lambda \in \sigma(F) \\ \re(-i\lambda)<0}} \re(-i\lambda) r_\lambda =  \sum_{j = 1}^{d} \im \lambda_j.
	\end{equation}
\end{remark}

\begin{remark}\label{rem-fractional}
	Assuming the setting of Theorem \ref{thm-gabmatdec}, let $\eta_t \geq 0$ denote the distribution density function of the $\nu$-stable subordinator at time $t$ \cite{bogdan_09}. It is a one-sided subordinator by construction, that is $\eta_t(s)=0$ for $s\leq 0$, and we have the identity
	\begin{equation}\label{eq-subord-prop}
		\irp e^{-us} \, \eta_t(s) \, \dd{s}=e^{-tu^{\nu}}, \qquad  u \ge 0.
	\end{equation} 
	This numerical identity can be used to introduce the ``fractional semigroup'' $T_{q,\nu}(t)$, with $t\ge 0$ and $0<\nu<1$, initially defined on $\cS(\rd)$ by setting 
	\begin{equation}
		T_{q,\nu}(t) f(y) \coloneqq \int_0^{+\infty} e^{sq^\w} f(y) \eta_t(s) \dd{s}, \qquad \forall y \in \rd. 
	\end{equation} 
	Although this definition is somehow artificial, it is consistent with the representations of the semigroup $e^{t(q^\w)^\nu}$ associated with the fractional powers of $q^\w$, whenever the latter can be meaningfully defined (e.g., via Bochner subordination or Balakhrishnan's formula --- as in the case of the harmonic oscillator).
	
In any case, the Gabor matrix of $T_{q,\nu}(t)$ satisfies
		\begin{align}
		\abs{\lan T_{q,\nu}(t) \pi(z)g,\pi(w)\gamma \ran} & \le \int_0^{+\infty} 	\abs{\lan e^{sq^\w} \pi(z)g,\pi(w)\gamma \ran} \eta_t(s) \dd{s} \\
		& \lesssim \lan w-z \ran^{-2N} \int_0^{+\infty} e^{-s\mu} \eta_t(s) \dd{s} \\
		& = e^{-t\mu^{\nu}} \lan w-z \ran^{-2N}. 
	\end{align}
	As a consequence, $T_{q,\nu}(t)$ extends to a bounded operator on every $M^{p,q}(\rd)$, $1 \le p,q \le \infty$, satisfying
	\begin{equation}
		\norm{T_{q,\nu}(t)}_{M^{p,q} \to M^{p,q}} \lesssim e^{-t\mu^\nu}. 
	\end{equation}
	Some examples in the fractional scenario are discussed in Section \ref{sec-examples} below.
\end{remark}

\subsection{Proof of Theorem \ref{thm-main-mod}}

In order to give a proof of Theorem \ref{thm-main-mod}, it is necessary to give a precise account of the symplectic decomposition of the symbol $q$ anticipated in the Introduction. 

\begin{proposition}[{\cite{hitrik}}]\label{prop-split}
	Let $q$ be a complex quadratic form on $\rdd$ with a non positive real part, namely $\re q \le 0$. Let $Q \in \bC^{2d,2d}$ be the companion symmetric matrix and $F=JQ$ the corresponding Hamilton map.

Let $S= \bigcap_{j=1}^{2d-1} \ker (\re F(\im F)^j)\cap \rdd$ be the singular space of $q$. Assume that $S$ is a symplectic subspace and $S \ne \rdd$. 
	
There exists a symplectic transformation $\chi \in \Spdr$ such that 
\begin{equation}
\chi(x,\xi) = (x',x'';\xi',\xi'') \in \rdd,
\end{equation} where $(x',\xi') \in \rnnp$ and $(x'',\xi'')\in \rnns$ are symplectic coordinates on $\Sbots$ and $S$ respectively --- hence $\rdd = \Sbots \oplus^{\bots} S$.  

In particular, we have the decomposition
\begin{equation}
	(q\circ \chi)(x,\xi) = q_1(x',\xi') + iq_2(x'',\xi''),
\end{equation} where:
\begin{itemize}
	\item $q_1= \restr{q}{\Sbots}$ is the complex quadratic form on $\rnnp$ defined by 
	\begin{equation}
		q_1(x',\xi') = \sigma((x',\xi'), F_1 (x',\xi')),
	\end{equation}
	where $F_1 = \restr{F}{\Sbots}$ is the corresponding Hamilton map. We have that $\re q_1 \le 0$ and the singular space $S_1$ associated with $q_1$ is trivial, that is $S_1 = \{0\}$.
	
	\item $q_2$ is the real quadratic form on $\rnns$ defined by
	\begin{equation}
		q_2(x'',\xi'') = \sigma((x'',\xi''), \im F_2(x'',\xi'')),
	\end{equation} where $F_2 = \restr{F}{S}$ is the Hamiltonian map of $\restr{q}{S}$. 
\end{itemize}
\end{proposition}

\begin{remark}\label{rem-metap}
	Assuming the setting of Proposition \ref{prop-split}, the symplectic covariance of Weyl calculus \eqref{eq-sympcov} implies that 
	\begin{equation}
		(q\circ \chi)^\w = \mu(\chi)^{-1}q^\w \mu(\chi),
	\end{equation} where $\mu(\chi)$ is a metaplectic operator associated with $\chi$. The companion generated semigroups then satisfy
	\begin{equation}
		e^{t(q\circ\chi)^\w} = \mu(\chi)^{-1} e^{tq^\w} \mu(\chi), \qquad t\ge 0. 
	\end{equation} 
	With the notation of Proposition \ref{prop-split}, since $\rdd = \rnnp \otimes \rnns$ we have the tensorization 
	\begin{equation}
		e^{t(q\circ \chi)^\w} = (e^{tq_1^\w} \otimes I)(I \otimes e^{itq_2^\w}), 
	\end{equation} hence
	\begin{equation}
		e^{tq^\w} = \mu(\chi)^{-1}(e^{tq_1^\w} \otimes I)(I \otimes e^{itq_2^\w}) \mu(\chi). 
	\end{equation}
\end{remark}

We are now ready to give a proof of Theorem \ref{thm-main-mod}, which is divided in three steps for the sake of exposition. Other equivalent approaches can be considered as well (e.g., Schur-type $M^1-M^\infty$ interpolation as in the proof of \cite[Theorem 3.3]{T_tams}). 

\begin{proof}[Proof of Theorem \ref{thm-main-mod}]
	\noindent \textbf{Step 1.} \textit{The Gabor matrix of $e^{tq^\w}$.} 
	
Consider the setting of Proposition \ref{prop-split} with $n'=d-n$ and $n''=n$. Let $g,\gamma \in \cS(\rd)\smo$ to be determined later, and $z,w \in \rdd$. In light of Remark \ref{rem-metap}, the Gabor matrix of $e^{tq^\w}$ reads
\begin{equation}
	K_{e^{tq^\w}}^{(\gamma,g)}(w,z) \coloneqq \lan e^{tq^\w} \pi(z)g, \pi(w) \gamma \ran = \lan e^{t(q\circ \chi)^\w} \mu(\chi) \pi(z) g, \mu(\chi) \pi(w)\gamma \ran. 
\end{equation} The intertwining property of metaplectic operators \eqref{eq-metap-intertw}, namely 
\begin{equation}
	\pi(\chi z) = c_\chi(z) \mu(\chi) \pi(z) \mu(\chi)^{-1}, 
\end{equation} for some $c_\chi(z) \in \bC$ with $|c_\chi(z)|=1$, implies 
\begin{equation}
	\abs{\lan e^{tq^\w} \pi(z)g, \pi(w) \gamma \ran} = \abs{\lan e^{t(q\circ \chi)^\w} \pi(\chi z) \mu(\chi) g, \pi(\chi w) \mu(\chi) \gamma \ran}.
\end{equation}
In light of the splitting $\rdd = \Sbots \oplus^{\bots} S = \rnnp \otimes \rnns$ induced by $\chi$, it is non restrictive to choose $g = \mu(\chi)^{-1}(g'\otimes g'')$ and $\gamma = \mu(\chi)^{-1}(\g'\otimes \g'')$ for arbitrarily chosen $g',\gamma' \in \cS(\rnp)\smo$ and $g'',\gamma'' \in \cS(\rns)\smo$. 

Straightforward computations show that the Gabor matrix of the composition of operators $A,B$ satisfies
\begin{equation}
	\abs{K_{AB}^{(\mu(\chi)\g,\mu(\chi)g)}(\chi w,\chi z)} \le \irdd \abs{K_{A}^{(\mu(\chi)\g,h)}(\chi w,u)} \abs{K_{B}^{(h,\mu(\chi)g)}(u,\chi z)} \dd{u},
\end{equation} for every $h \in \cS(\rdd)\smo$. This is precisely the situation under our attention, with $A=e^{tq_1^\w}\otimes I$ and $B=I \otimes e^{itq_2^\w}$. 

Let us focus on $K_{A}$ first. We conveniently choose $h=(h' \otimes h'')$ for some $h' \in \cS(\rnp)\smo$ and $h'' \in \cS(\rns)\smo$, and write $\chi y = (\tilde y', \tilde y'') \in \rnnp \times \rnns$ for every $y \in \rdd$, so that 
\begin{equation}
	\pi(\chi y)(h'\otimes h'') = \pi(\tilde y')h' \otimes \pi(\tilde y'')h''. 
\end{equation}
Therefore, for every $M,N \in \bN$ and $u=(u',u'') \in \rnnp \times \rnns = \rdd$ we have
\begin{align}
 \abs{K_{A}^{(\mu(\chi)\g,h)}(\chi w,u)} & = \abs{\lan (e^{tq_1^\w} \otimes I) \pi(u', u'')(h'\otimes h'') \ran, \pi(\tilde w',\tilde w'')(\g'\otimes \g'')} \\
 & = \abs{\lan e^{tq_1^\w} \pi(u') h' ,\pi(\tilde w')\g'\ran} \abs{\lan \pi(u'') h'',\pi(\tilde w'')\g'' \ran} \\
 & \lesssim_{n',N,M,q} e^{-t\mu'} \lan \tilde w'- u' \ran^{-2N} \lan \tilde w'' - u'' \ran^{-M},
\end{align} where in the last step we resorted to Theorem \ref{thm-gabmatdec} and Lemma \ref{lem-stft}, with $\mu'$ being defined accordingly. 

Consider now the Gabor matrix of $B$. For every $M,N \in \bN$ we have
\begin{align}
	\abs{K_{B}^{(h,\mu(\chi)g)}(u,\chi z)}& = \abs{\lan (I \otimes e^{itq_2^\w})\pi(\tilde z',\tilde z'')(g'\otimes g'') , \pi(u',u'')(h'\otimes h'') \ran} \\
	& = \abs{\lan \pi(\tilde z')g',\pi(u') h' \ran} \abs{\lan e^{itq_2^\w} \pi(\tilde z'')g'',\pi(u'') h'' \ran} \\
	& \lesssim_{n'',M,N} \lan u'- \tilde z' \ran^{-2N} (\det \Sigma_t'')^{-1/2} \lan D_tU_t(u''-H_t \tilde z'')\ran^{-M},
\end{align} where in the last step we resorted to Lemma \ref{lem-stft}  and Theorem \ref{thm-gabmatdisp}, $(U_t,V_t,\Sigma_t'')$ being any Euler decomposition of the classical flow $H_t \in \Sp(n'',\bR)$ associated with $q_2$. 

To sum up, we have
\begin{multline}
		\abs{K_{e^{tq^\w}}^{(\g,g)}(w,z)} \le \irdd \abs{K_{A}^{(\mu(\chi)\g,h)}(\chi w,u)} \abs{K_{B}^{(h,\mu(\chi)g)}(u,\chi z)} \dd{u} \\
	\lesssim_{d,N,M,q} e^{-t\mu'} (\det \Sigma_t'')^{-1/2} \Big( \int_{\rnnp}  \lan \tilde w'- u' \ran^{-2N} \lan u'- \tilde z' \ran^{-2N} \dd{u'} \Big) \\ \times \Big( \int_{\rnns} \lan \tilde w'' - u'' \ran^{-M} \lan D_tU_t(u''-H_t \tilde z'')\ran^{-M} \dd{u''} \Big). 
\end{multline}
Concerning the first integral, it suffices to choose $N>n'$ to obtain by subconvolutivity (cf.\ \eqref{eq-subconv}) that
\begin{equation}
 \int_{\rnnp}  \lan \tilde w'- u' \ran^{-2N} \lan u'- \tilde z' \ran^{-2N} \dd{u'} \lesssim_N \lan \tilde w' - \tilde z' \ran^{-2N}. 
\end{equation}
Concerning the second integral,
since $U_t$ is an orthogonal matrix we have
\begin{multline}
	\int_{\rnns} \lan \tilde w'' - u'' \ran^{-M} \lan D_tU_t(u''-H_t \tilde z'')\ran^{-M} \dd{u''} \\ = \int_{\rnns} \lan U_t \tilde u'' \ran^{-M} \lan D_tU_t(u''+\tilde w'' -H_t \tilde z'') \ran^{-M} \dd{u''} \\ 
	 = \int_{\rnns} \lan \tilde u'' \ran^{-M} \lan D_tu''+D_tU_t(\tilde w'' -H_t \tilde z'') \ran^{-M} \dd{u''} \\ \lesssim_M \lan D_tU_t(\tilde w'' - H_t \tilde z'') \ran^{-M},
\end{multline} provided that $M>n''$, where in the last step we have made repeated use of \cite[Lemma 2.6]{CNT_dispdiag20}. Then
\begin{equation}\label{eq-gabmat-mix}
	\abs{K_{e^{tq^\w}}^{(\g,g)}(w,z)} \lesssim_{d,N,M,q} e^{-t\mu'} (\det \Sigma_t'')^{-1/2} \lan \tilde w' - \tilde z' \ran^{-2N} \lan D_tU_t(\tilde w'' - H_t \tilde z'') \ran^{-M}.  
\end{equation}

\noindent \textbf{Step 2.} \textit{Continuity on modulation spaces.}

Studying the continuity $e^{tq^\w}$ on $M^p(\rd)$ is equivalent to investigating the boundedness on $L^p(\rdd)$ of the integral operator
\begin{equation}
	T F(w) \coloneqq \irdd K_{e^{tq^\w}}^{(\g,g)}(w,z) F(z) \dd{z}. 
\end{equation} To this aim, if we were able to find $C(t)>0$ such that
\begin{equation}
	\irdd \abs{K_{e^{tq^\w}}^{(\g,g)}(w,z)} \dd{w} \le C(t), \qquad \irdd \abs{K_{e^{tq^\w}}^{(\g,g)}(w,z)} \dd{z} \le C(t),
\end{equation} then by Young's integral inequality we would infer
\begin{align}
	\| TF \|_{L^p} & =  \Big( \irdd \abs{ \irdd K_{e^{tq^\w}}^{(\g,g)}(w,z) F(z) \dd{z}}^p \dd{w}\Big)^{1/p} \\
	& \le C(t) \|F\|_{L^p}. 
\end{align}
Choose $N=M/2$. Since  $D_tU_t H_t =E_tV_t$, with $E_t= (I \oplus (\Sigma_t'')^{-1})$, then
\begin{equation}
	(1+\abs{\tilde w' - \tilde z'})^{-M} (1+\abs{D_tU_t(\tilde w'' - H_t \tilde z'')})^{-M} \le  (1+\abs{\tilde w' - \tilde z'} + \abs{D_t U_t \tilde w'' - E_t V_t \tilde z''})^{-M},
\end{equation} and the substitution $y = (I\oplus D_tU_t)\chi w - (I \oplus E_tV_t)\chi z$ yields
\begin{equation}
	\max\Big\{ \irdd \abs{K_{e^{tq^\w}}^{(\g,g)}(w,z)} \dd{w}, \irdd \abs{K_{e^{tq^\w}}^{(\g,g)}(w,z)} \dd{z} \Big\} \lesssim e^{-t\mu'} (\det \Sigma_t'')^{1/2},
\end{equation} provided that $M>2d$. To conclude, for every $1 \le p \le \infty$ and $f \in M^p(\rd)$, we have 
\begin{equation}
	\| e^{tq^\w} \|_{M^p} \lesssim e^{-t\mu'}(\det \Sigma_t'')^{1/2} \| f\|_{M^p}. 
\end{equation} 
 
 \noindent \textbf{Step 3.} \textit{Refining the dispersive dependence.}
 
 To conclude, note that the estimate above is certainly not optimal in the case where $p=2$, due to the appearance of the dispersive factor. On the other hand, since $M^2(\rd)=L^2(\rd)$ with equivalence of norms and $L^2(\rd)= L^2(\rnp;L^2(\rns))$, in this case we have
 \begin{align}
 	\|e^{tq^\w}\|_{L^2 \to L^2} & = \| \mu(\chi)^{-1} e^{t(q\circ \chi)^\w} \mu(\chi) \|_{L^2 \to L^2} \\
 	& = \| e^{tq_1^\w} \otimes e^{itq_2^\w}\|_{L^2\to L^2} \\
 	& = \| e^{tq_1^\w}\|_{L^2 \to L^2} \\
 	& \lesssim e^{-t\mu'}.  
 \end{align}
As far as complex interpolation is concerned, modulation spaces behave like $L^p$ spaces (cf.\ \cite[Proposition 2.3.17]{CR_book}), so interpolating $M^1-M^2$ and $M^2- M^\infty$ yields the claim. 
\end{proof}

\begin{remark}
	We emphasize that boundedness results for $e^{tq^\w}$ on $M^p(\rd)$ are not expected to extend to weighted or mixed modulation spaces $M^{p,q}(\rd)$ in general, since metaplectic operators (involved in both the symbol decomposition and in $e^{itq_2^\w}$) generally fail to be bounded there --- see \cite{FS} for precise characterizations in this regard.
\end{remark}

\section{Some examples} \label{sec-examples}

Motivated by the examples discussed in \cite{hitrik,pravda_18}, we briefly present a number of applications of the previous results. 

\subsection{A globally elliptic model: the Hermite operator}
Consider the heat equation for the harmonic oscillator on $\rd$, that is
\begin{equation}
	\partial_t u(t,x)= (\Delta_x - x^2)u(t,x), \qquad (t,x) \in [0,+\infty) \times \rd, 
\end{equation} so that $q(x,\xi)=-(x^2+\xi^2)$. The associated matrices are thus given by $Q=-I$ and $F=JQ=-J$. It is then clear that $F$ is diagonalizable, the eigenvalues being $\pm i$ (each with multiplicity $d$), and the singular space of $q$ is trivial:
\begin{equation}
	S= \ker (\re F) = \{0\}. 
\end{equation}
The Mehler formula \eqref{eq-intro-mehler} thus reads
\begin{equation}
	\Theta_t(y) = (\cosh t)^{-d}e^{-(\tanh t)y^2}, \quad y \in \rdd.
\end{equation}
We are in the position to consider direct application of Theorem \ref{thm-gabmatdec}, where $\mu=d$, 
hence the Gabor matrix of $e^{tq^\w}$ satisfies
\begin{equation}
	\abs{\lan e^{tq^\w} \pi(z)g,\pi(w)\gamma \ran} \lesssim e^{-td} \lan w-z \ran^{-2N}, \qquad N \in \bN,
\end{equation} 
from which we infer boundedness of $e^{tq^\w}$ on every modulation space $M^{p,q}(\rd)$, $1 \le p,q \le \infty$, with operator norm $\lesssim e^{-td}$. This result improves the findings in \cite{cordero_21} and is consistent with those in \cite{bhimani_21} --- the latter have in fact a significantly broader scope (i.e., boundedness $M^{p_1,q_1}\to M^{p_2,q_2}$) since they exploit the special pseudo-differential structure of $e^{tq^\w}$. 

The case of fractional powers of $q^\w$, trated in \cite{bhimani_21}, can be approached via subordination as detailed in Remark \ref{rem-fractional}, so that for every $0<\nu<1$ we have
\begin{equation}
	\abs{\lan e^{t(q^\w)^\nu} \pi(z)g,\pi(w)\gamma \ran} \lesssim e^{-td^\nu} \lan w-z \ran^{-2N}, \qquad N \in \bN.
\end{equation} 

\subsection{A degenerate elliptic model: the special Hermite operator}
The twisted Laplacian, also known as the special Hermite operator, is the differential operator on $\rdd$ defined by 
\begin{equation}
	\cL = q^\w, \qquad q(z,\zeta)=- \sum_{j=1}^d \Big[\Big(\xi_j - \frac{y_j}{2} \Big)^2 + \Big(\eta_j+\frac{x_j}{2} \Big)^2\Big], \quad z=(x,y), \, \zeta=(\xi,\eta) \in \rdd.
\end{equation}
The associated matrices are
\begin{equation}
	Q= \begin{bmatrix} -I/4 & -J/2 \\ J/2 & -I \end{bmatrix}, \qquad F = JQ = \begin{bmatrix} J/2 & -I \\ I/4 & J/2 \end{bmatrix}. 
\end{equation}
The singular space of $q$ is non-trivial, given by
\begin{equation}
	S= \ker F = \{ (z,Jz/2) : z \in \rdd \}, 
\end{equation} and inherits the symplectic structure --- we have that
\begin{equation}
	\Sbots = \{ (z,-Jz/2) : z \in \rdd\}, \qquad S \cap \Sbots = \{0\}. 
\end{equation} Note that $\cL$ fails to be elliptic even when restricted to $S$. 

The non-triviality of the singular space reflects into the fact that, although $F$ is diagonalizable,  we cannot directly apply Theorem \ref{thm-gabmatdec} since its eigenvalues are $0$ and $\pm i$. This effect is clearly evident when the Mehler formula \eqref{eq-intro-mehler} for $e^{tq^\w}$ is taken into account:
\begin{equation}
	e^{t\cL} = \Theta_t^\w, \qquad \Theta_t(y) = (8\pi^2 \cosh t)^{-d}e^{-(\tanh t)|y_2-Jy_1/2|^2}, \quad y=(y_1,y_2) \in \rdd \times \rdd. 
\end{equation} In particular, note the lack of decay of the symbol when $y \in S$. 

Although slight modifications of the argument of Theorem \ref{thm-gabmatdec} would readily give the same conclusion, given the peculiar form of the symbol $\Theta_t$, let us stick to the frameworks of Theorem \ref{thm-main-mod}. It is easy to realize that $\restr{Q}{\Sbots} = -I_{2d}$, hence $\restr{F}{\Sbots}=-J$, having $\pm i$ as eigenvalues with multiplicity $d$ each. As a result, for $g,\gamma \in \cS(\rd)\smo$ and $z,w \in \rdd$ we obtain 
\begin{equation}
	\abs{ \lan e^{t\restr{q^\w}{\Sbots}} \pi(z)g,\pi(w)\gamma \ran} \lesssim e^{-td} \lan w-z \ran^{-2N}, \qquad N \in \bN,
\end{equation} and thus, since $\im F=0$, Theorem \ref{thm-main-mod} yields 
\begin{equation}
	\norm{e^{tq^\w} f}_{M^p} \lesssim e^{-td} \norm{f}_{M^p}
\end{equation} for every $1 \le p \le \infty$, in accordance with the results obtained in \cite{T_tams} --- which are actually more powerful, since the special twisted structure of the operator $\cL$ is exploited. The contents of Remark \ref{rem-fractional} apply here, showing boundedness on $M^p$ of the fractional semigroup $e^{-t\cL^\nu}$, $0 < \nu < 1$, with operator norm $\lesssim e^{-td^\nu}$. 

\subsection{A non-elliptic model: the Kramers-Fokker-Planck operator} A prominent example of non-selfadjoint, non-elliptic quadratic model is the second-order Kramers-Fokker-Planck operator with quadratic potential. In fact, the pioneering analysis of such operator in \cite{helffer,herau} stimulated the subsequent theory of singular spaces. 

We are focusing on the quadratic operator defined on $\bR^2$ by
\begin{equation}
	\cK=-\Delta_v +\frac14 v^2 -\frac12 +v\partial_x-(\partial_xV(x))\partial_v = -q^\w-\frac12, \qquad (x,v) \in \bR^2, 
\end{equation} where $V(x)=ax^2/2$ for some $a \in \bR\smo$ and 
\begin{equation}
	q(x,v,\xi,\eta) = -\eta^2 -\frac14 v^2-i(v\xi-ax\eta).
\end{equation}
The Hamilton matrix associated with $q$ is
\begin{equation}
	F = \begin{bmatrix}
		0 & -i/2 & 0 & 0 \\ ia/2 & 0 & 0 & -1 \\ 0 & 0 & 0 & -ia/2 \\ 0 & 1/4 & i/2 & 0
	\end{bmatrix}. 
\end{equation} It is not difficult to check that the singular space of $q$ is trivial, and that $F$ is diagonalizable if $a\ne 1/4$. In that case, the eigenvalues $\lambda_1,\lambda_2$ with positive imaginary part are easily computed:
\begin{equation}
	\lambda_1= (-1+\sqrt{1-4a})\frac{i}{4}, \qquad 	\lambda_2= (1+\sqrt{1-4a})\frac{i}{4} \qquad (a<0),
\end{equation} 
\begin{equation}
	\lambda_1= (1-\sqrt{1-4a})\frac{i}{4}, \qquad 	\lambda_2= (1+\sqrt{1-4a})\frac{i}{4} \qquad (a>0, \, a \ne \tfrac14).
\end{equation} 
Theorem \ref{thm-gabmatdec} thus applies to $e^{tq^\w}$ with 
\begin{equation}
	\mu = \im{\lambda_1} + \im{\lambda_2} = \begin{cases} \frac{\sqrt{1-4a}}{2} & (a<0) \\ \frac12, & (a>0, \, a \ne \tfrac14) \end{cases} 
\end{equation}
from which we infer, for every $1 \le p \le \infty$,
\begin{equation}
	\norm{e^{-t\cK}}_{M^p \to M^p} = e^{t/2}\norm{e^{tq^\w}}_{M^p \to M^p} \lesssim \begin{cases} e^{-t \frac{\sqrt{1-4a}-1}{2}} & (a<0) \\ 1 & (a>0, \, a \ne \tfrac14). \end{cases}
\end{equation}

\subsection{Quadratic operators with non-trivial real and imaginary parts}

Consider the equation
\begin{equation}
	\partial_t u(t,x) = (2i\Delta_x - x^2) u(t,x), \qquad (t,x) \in [0,+\infty) \times \rd,
\end{equation} which is of quadratic type with symbol $q(x,\xi)=-(x^2+2i\xi^2)$. The Hamilton matrix is
\begin{equation}
	F = \begin{bmatrix}
		O & -2i I \\ I & O
	\end{bmatrix}, 
\end{equation} hence diagonalizable with eigenvalues $\pm (-1+i)$, each with multiplicity $d$. To compute the singular space we note that 
\begin{equation}
	\ker (\re F) = \{0\} \times \rd, \qquad \ker (\re F \im F) = \rd \times\{0\},
\end{equation} hence we infer $S=\{0\}$. Theorem \ref{thm-gabmatdec} thus gives
\begin{equation}
	\abs{\lan e^{tq^\w} \pi(z)g,\pi(w)\gamma \ran} \le C e^{-td} (1+\abs{w-z})^{-2N}. 
\end{equation}

A simple recipe to produce less trivial examples of joint diffusive/dispersive effects, which somehow mimics the decomposition of Proposition \ref{prop-split}, is described in \cite{hitrik,pravda_11}. Let $q_1(x_1,\xi_1)$ be a complex quadratic form on $\bR^{2d_1}$ with $\re q_1 \le 0$ and trivial singular space $S_1=\{0\}$. Consider now a real quadratic form $q_2(x_2,\xi_2)$ on $\bR^{2d_2}$. Then it is not difficult to show that the form 
\begin{equation}\label{eq-qform-make}
	q(x_1,x_2,\xi_1,\xi_2) = q_1(x_1,\xi_1) + i q_2(x_2,\xi_2),
\end{equation} defined on $\bR^{2(d_1+d_2)}$, has a non-trivial singular space given by
\begin{equation}
	S=\{(0,a,0,b) : a,b \in \bR^{d_2} \}, 
\end{equation} hence $\Sbots = \{(u,0,v,0) : u,v \in \bR^{d_1} \}$. 

Let us discuss an example in this spirit, by combining an oscillator-type heat component with a free particle Schr\"odinger dynamics:
\begin{equation}
	q(x_1,x_2,\xi_1,\xi_2) = -(x_1^2 + \xi_1^2) +i\xi_2^2. 
\end{equation}
The Hamilton matrix is thus given by
\begin{equation}
	F= \begin{bmatrix}
		O_{d_1+d_2} & -I_{d_1} \oplus i I_{d_2} \\ I_{d_1} \oplus O_{d_2} & O_{d_1+d_2}
	\end{bmatrix},
\end{equation} and we have 
\begin{equation}
	\restr{F}{\Sbots} = -J_{d_1}, \qquad \im \restr{F}{S} = \begin{bmatrix}
		O_{d_2} & I_{d_2} \\ O_{d_2} & O_{d_2}
	\end{bmatrix}.
\end{equation}
Then, for $t \ge 0$,  
\begin{equation}
	e^{2t\im \restr{F}{S}} = I + 2t\im \restr{F}{S} = \begin{bmatrix}
	I_{d_2} & 2t I_{d_2} \\ O_{d_2} & I_{d_2}
	\end{bmatrix},
\end{equation} whose $d_2$ largest singular values coincide: $\sigma_1(t) = \ldots = \sigma_{d_2}(t) \asymp (1+t)$. To conclude, for every $1 \le p \le \infty$ we have
\begin{equation}
	\norm{e^{tq^\w}}_{M^p \to M^p} \lesssim e^{-td_1} (1+t)^{d_2\abs{\frac12 - \frac1p}}. 
\end{equation} 

\section*{Acknowledgements} 
	We gratefully acknowledge helpful discussions on the topics of this note with Fabio Nicola, Federico Stra and Patrik Wahlberg. 
	
	The author is member of Gruppo Nazionale per l’Analisi Matematica, la Probabilit\`a e le loro Applicazioni (GNAMPA) --- Istituto Nazionale di Alta Matematica (INdAM). The present research has been developed as part of the activities of the GNAMPA-INdAM project ``Analisi spettrale, armonica e stocastica in presenza di potenziali magnetici'', award number (CUP): E5324001950001. 
	
	The author reports that there are no competing interests to declare. No new data were generated or analysed in support of this research.

\end{document}